\documentclass{article}
\usepackage{amsmath}
\usepackage{amssymb}
\usepackage{amsthm}
\usepackage{algorithm}
\usepackage{algorithmic}

\newtheorem{Definition}{Definition}[subsection]
\newtheorem{Theorem}[Definition]{Theorem}
\newtheorem{Lemma}[Definition]{Lemma}
\newtheorem{Corollary}[Definition]{Corollary}
\newtheorem{Example}[Definition]{Example}
\newtheorem{Remark}[Definition]{Remark}
\newtheorem{Proposition}[Definition]{Proposition}
\newcommand{\Ass}{{\mathop{\mathrm{Ass}}}}
\newcommand{\hull}{{\mathop{\mathrm{hull}}}}

\newcommand{\lc}{{\mathop{\mathrm{lc}}}}

\begin{document}
\title{Modular Techniques for Effective Localization and Double Ideal Quotient}
\author{Yuki Ishihara \thanks{Graduate School of Science, Rikkyo University, 3-34-1 Nishi-Ikebukuro, Toshima-ku, Tokyo, Japan, 171-8501, yishihara@rikkyo.ac.jp}}
\date{}
\maketitle
\begin{abstract}
By double ideal quotient, we mean $(I:(I:J))$ where $I$ and $J$ are ideals. In our previous work [11], double ideal quotient and its variants are shown to be very useful for checking prime divisors and generating primary components. Combining those properties, we can compute "direct localization" effectively, comparing with full primary decomposition. In this paper, we apply modular techniques effectively to computation of such double ideal quotient and its variants, where first we compute them modulo several prime numbers and then lift them up over rational numbers by Chinese Remainder Theorem and rational reconstruction. As a new modular technique for double ideal quotient and its variants, we devise criteria for output from modular computations. Also, we apply modular techniques to intermediate primary decomposition. We examine the effectiveness of our modular techniques for several examples by preliminary computational experiments in Singular. 
\end{abstract}

\section{New Contributions}
For proper ideals $I$ and $J$, {\em double ideal quotient} is an ideal of shape $(I:(I:J))$. It and its variants are effective for localization and give us criteria for prime divisors (primary components) and ways to generate primary components. In  \cite{Ishi-Yoko},  "Local Primary Algorithm" computes the specific primary component from given a prime ideal without full primary decomposition. However, they tend to be very time-consuming for computing Gr\" obner bases and ideal quotients in some cases. Also, there is another problem with a way to find candidates of prime divisors. As a solution of these problems, we propose a new method for computing double ideal quotient in the $n$ variables polynomial ring with rational coefficients $\mathbb{Q}[X]$ by using "Modular Techniques", where $X=\{x_1,\ldots,x_n\}$. It is well-known that modular techniques are useful to avoid intermediate coefficient growth and have a good relationship with parallel computing (see \cite{Afzal, Bohm, Idrees, Noro2009}). In this paper, we have the following contributions. 
\begin{enumerate}
	\item Apply modular techniques to double ideal quotient. (Theorem \ref{theorem:diq} and Theorem \ref{assTest})
	\item Extend criteria about prime divisor in \cite{Ishi-Yoko}.  (Theorem \ref{cri:rad})
	\item Devise a new method for certain intermediate decomposition in some special cases. (Corollary \ref{interpri}, Proposition \ref{interrad})
\end{enumerate} 

For a prime number $p$, let $\mathbb{Z}_{(p)}=\{a/b\in \mathbb{Q}\mid p \nmid b\}$ be the localized ring by $p$ and $\mathbb{F}_p[X]$ the polynomial ring over the finite field. We denote by $\phi_p$  the canonical projection $\mathbb{Z}_{(p)}[X]\to \mathbb{F}_p[X]$. Given ideals $I$ and $J$ in the polynomial ring with rational coefficients $\mathbb{Q}[X]$, we first compute double ideal quotient of the image $\phi_p((I:(I:J))\cap \mathbb{Z}_{(p)}[X])$ in  $\mathbb{F}_p[X]$ for  "lucky" primes $p$ (we will discuss such luckiness later). Next, we lift them up to $G_{can}$, a candidate of Gr\"obner basis, from the computed Gr\"obner basis $\bar{G}$ of $\phi_p((I:(I:J))\cap \mathbb{Z}_{(p)}[X])$ by using Chinese Remainder Theorem (CRT) and rational reconstruction (see \cite{Bohm}). Avoiding intermediate coefficient growth, this method is effective for several examples. 

Also, we extend the criterion in \cite{Ishi-Yoko} about prime divisor in order to compute certain "intermediate decomposition" of ideals and to find prime divisors in some special cases. For an ideal $I$ and a prime ideal $P$, it follows that $P$ is a prime divisor of $I$ if and only if $P\supset (I:(I:P))$ (see Theorem 31 (Criterion 5), \cite{Ishi-Yoko}). However, the projected image of a prime ideal may not be a prime ideal but an intersection of prime ideals in $\mathbb{F}_p[X]$. Thus, we generalize the criterion to a radical ideal $J\supset I$; it follows that every prime divisor $P$ of $J$ is associated with $I$ if and only if $J\supset (I:(I:J))$. For such a radical ideal $J$, if $J$ is unmixed, we can compute the intersection  of primary components $Q$ of $I$ whose associated prime is a prime divisor of $J$ by modular techniques. This ideal may be considered as an "intermediate component" of $I$.  By gathering these intermediate components, we may obtain an "intermediate primary decomposition" (see Definition \ref{def:ipd}).  For this computation, we can utilize maximal independent sets (see Section \ref{sec:ipd} for the definition of maximal independent set). 

Primary decomposition of an ideal in a polynomial ring over a field is an essential tool of Commutative Algebra and Algebraic Geometry. Algorithms of primary decomposition have been much studied, for example, by \cite{GIANNI1988149, SHIMOYAMA1996247, KAWAZOE20111158, Eisenbud1992}. We apply double ideal quotient to check whether candidates of prime divisors from modular techniques are associated with the original ideal or not. It shall contribute the total efficiency of the whole process since our "intermediate decomposition" may divide a big task into small ones as a "divide-and-conquer" strategy. 

This paper is organized as follows. In section 2.1, we introduce extended criteria for prime divisor and primary component based on double ideal quotient and its variants. In section 2.2, we apply modular techniques to double ideal quotient and its variants. In section 2.3, we sketch an outline of intermediate primary decomposition. In section 3, we see some effectiveness of modular method in several examples in a preliminary experiment. Its practicality will be examined by computing more detailed experiments. 

\section{Main Theorems}
Here we show theoretical bases for our new techniques described in Section 1. We denote an arbitrary field by $K$ and the ideal  generated by $f_1,\ldots,f_s\in K[X]$ by $(f_1,\ldots,f_s)_{K[X]}$. If the base ring is obvious, we simply write $(f_1,\ldots,f_s)$. Also, we denote by $K[X]_P$ the localized ring by a prime ideal $P$ and by $I_P$ the ideal $IK[X]_P$ respectively.  For an irredundant primary decomposition $\mathcal{Q}$ of $I$, we say that $P$ is a prime divisor of $I$ if there is a primary component $Q\in \mathcal{Q}$ s.t. $P=\sqrt{Q}$. For simplicity, we assume primary decomposition is irredundant. We denote the set of prime divisors of $I$ by $\Ass (I)$ (see Definition 4.1.1 and Theorem 4.1.5 in \cite{greuel2002singular}). For a prime ideal $P$, $P_P$ is also prime. 

\subsection{Criteria for prime divisors and primary components}

First, we recall criteria using double ideal quotient and its variants (see \cite{Ishi-Yoko}, Sect. 3.). In Proposition \ref{cri1}, the equivalence between $(A)$ and $(B)$ is originally described in \cite{vasconcelos2004computational}. In our previous work \cite{Ishi-Yoko}, we relate it with a variant of double ideal quotient $(I:(I:P^{\infty}))$. Double ideal quotient is also used to compute equidimensional hull in \cite{Eisenbud1992}, which we use in Lemma \ref{hullpm_gen} and Theorem \ref{thm-p} later. We will show that such double ideal quotient(s) can be computed efficiently by modular techniques in Section 2.2.

\begin{Proposition}[\cite{Ishi-Yoko},  Theorem 31] \label{cri1}
	Let $I$ be an ideal and $P$ a prime ideal. Then, the following conditions are equivalent. 
			
			$(A)$ $P\in \Ass (I)$, 
			
			$(B)$ $P\supset (I:(I:P))$, 
			
			$(C)$ $P\supset (I:(I:P^{\infty}))$.
\end{Proposition}

\begin{Remark} \label{rem:IP}
	In Proposition \ref{cri1}, the condition $P\supset (I:(I:P))$ is equivalent to $P=(I:(I:P))$ since $P\subset (I:(I:P))$ always holds for any ideals $I$ and $P$. Indeed, $P(I:P)\subset I$ from the definition of $(I:P)$ and thus $P\subset (I:(I:P))$. 
\end{Remark}
\begin{Remark} \label{rem:com}
	The operations of double ideal quotient and localization by prime ideal are commutative. Indeed, for ideals $I$, $J$ and a prime ideal $P$, $(I:J)_P=(I_P:J_P)$ from Corollary 3.15 in \cite{atiyah1994introduction} and thus we obtain $(I:(I:J))_P=(I_P:(I:J)_P)=(I_P:(I_P:J_P))$. Similarly, we have $(I:(I:J^{\infty}))_P=(I_P:(I:J^{\infty})_P)=(I_P:(I_P:J^{\infty}_P))$ as $(I:J^{\infty})=(I:J^m)$ and $(I_P:J_P^{\infty})=(I_P:J_P^m)$ for a sufficiently large integer $m$. Also, a prime ideal $P$ is associated with an ideal $I$ if and only if $P_P$  is associated with $I_P$ since there is a correspondence between primary decompositions of $I$ and $I_P$  (see Proposition 4.9 in \cite{atiyah1994introduction}). Similarly, for a $P$-primary ideal $Q$, $Q$ is a $P$-primary component of $I$ if and only if $Q_P$ is a $P_P$-primary component of $I_P$. 
\end{Remark}

Next, we introduce extended theorems about double ideal quotient and its variants toward intermediate primary decomposition in Section 2.3. Proposition \ref{cri1} gives a relationship between an ideal $I$ and a prime divisor $P$. It can be extended to one between an ideal $I$ and an intersection of some prime divisors $J$. Thus, we consider a radical ideal $J$ instead of a prime ideal $P$ as follows.

\begin{Theorem} \label{cri:rad}
	Let $I$ be an ideal and $J$ a proper radical ideal. Then, the following conditions are equivalent. 
			
			$(A)$ $\Ass (J)\subset \Ass (I)$, 
			
			$(B)$ $J\supset (I:(I:J))$, 
			
			$(C)$ $J\supset (I:(I:J^{\infty}))$.
\end{Theorem}

\begin{proof}
	  First, we show that $(A)$ implies (B). Let $P\in \Ass (J)\subset \Ass (I)$. Then, $P\supset (I:(I:P))$ by Proposition  \ref{cri1}. Thus, $P\supset (I:(I:P))\supset  (I:(I:J))$. Since $J=\bigcap_{P\in \Ass (J)} P$, we obtain $J\supset (I:(I:J))$. Next, we show that  $(B)$ implies $(C)$. As $(I:J)\subset (I:J^{\infty})$, we obtain $J\supset (I:(I:J))\supset (I:(I:J^{\infty}))$. Finally, we show that  $(C)$ implies $(A)$. Let $P\in \Ass (J)$. Then, $J_P\supset (I:(I:J^{\infty}))_P= (I_P:(I_P:J_P^{\infty}))$ from Remark \ref{rem:com} and thus $J_P=P_P\in \Ass (I_P)$ from Proposition  \ref{cri1}. Hence,  $P\in \Ass (I)$ by Remark \ref{rem:com}. 
\end{proof}

\begin{Example}
	Let $I=(x)\cap (x^3,y)\cap (x^2+1)$ and $J=(x,y)\cap (x^2+1)$. Then, $(I:(I:J))=(x,y)\cap (x^2+1)=J$ and $\Ass (J)=\{(x,y),(x^2+1)\}\subset \Ass(I)=\{(x),(x,y),(x^2+1)\}$. In addition, we have $(I:(I:J^{\infty}))=(x^2,y)\cap (x^2+1)\subset J$. 
\end{Example}

To generate primary component, the following lemma is well-known. Here, for a $d$-dimensional ideal $I$, equidimensional hull $\hull (I)$ is the intersection of its $d$-dimensional primary components. 

\begin{Lemma}[\cite{Eisenbud1992}, Section 4. \cite{matzat2012algorithmic}, Remark 10]  \label{hullpm}
	Let $I$ be an ideal and $P$ a prime divisor of $I$. For a sufficiently large integer $m$, $\hull (I+P^m)$ is a $P$-primary component appearing in a primary decomposition of $I$. 
\end{Lemma}

Here, we generalize Lemma \ref{hullpm} to an intersection of equidimensional prime divisors as follows. 

\begin{Lemma} \label{hullpm_gen}
	Let $I$ be an ideal and $J$ an intersection of prime divisors of $I$. Suppose $J$ is unmixed i.e. $\dim (P)=\dim (J)$ for any $P\in \Ass (J)$. Then, for a sufficiently large integer $m$, $\hull (I+J^m)$ is an intersection of primary components appearing in a primary decomposition of $I$ i.e. $\hull (I+J^m)=\bigcap_{P\in \Ass (J)} Q(P)$ where $Q(P)$ is a $P$-primary component of $I$.
\end{Lemma}

\begin{proof}
	Let $m$ be a positive integer. First, we note that, for each $P\in \Ass (J)$, $I\subset \hull (I+J^{m})_P\cap K[X]\subset \hull (I+P^{m})$ since
	\begin{align*}
		I&\subset I+J^{m}\subset \hull (I+J^{m})\subset \hull (I+J^{m})_P\cap K[X]	\\
				 &\subset	\hull (I+P^{m})_P\cap K[X]=\hull (I+P^{m})
	\end{align*}
	 where the last equality comes from the fact that $\sqrt{I+P^{m}}=P$ and $P$ is the unique isolated prime divisor of $I+P^{m}$. By Lemma \ref{hullpm}, there exist a sufficiently large integer $m(P)$ and a primary decomposition $\mathcal{Q}$ of $I$ such that $\hull (I+P^{m(P)})\in \mathcal{Q}$. Then, 
	\[
	I\subset \bigcap_{P\in \Ass (J)} \hull (I+J^{m(P)})_P\cap K[X]\subset \bigcap_{P\in \Ass (J)}\hull (I+P^{m(P)})
	\]
	and, by intersecting $\bigcap_{Q\in \mathcal{Q}, \sqrt{Q}\not \in \Ass (J)} Q$ with them, we obtain
	\begin{align*}
		I&\subset \left(\bigcap_{P\in \Ass (J)} \hull (I+J^{m(P)})_P\cap K[X]\right)\cap \bigcap_{Q\in \mathcal{Q}, \sqrt{Q}\not \in \Ass (J)} Q		\\
		 &\subset \left((\bigcap_{P\in \Ass (J)}\hull (I+P^{m(P)})\right)\cap  \bigcap_{Q\in \mathcal{Q}, \sqrt{Q}\not \in \Ass (J)} Q=I. 
	\end{align*}
	Thus, $\left(\bigcap_{P\in \Ass (J)} \hull (I+J^{m(P)})_P\cap K[X]\right)\cap \bigcap_{Q\in \mathcal{Q}, \sqrt{Q}\not \in \Ass (J)} Q=I$ and $\hull (I+J^{m(P)})_P\cap K[X]$ is a $P$-primary component of $I$. 
Since $J$ is unmixed, $\sqrt{I+J^{m}}=\sqrt{J}=\bigcap_{P\in \Ass (J)}P$ and $\Ass (\hull (I+J^{m}))= \Ass (J)$ i.e.  $\hull (I+J^m)=\bigcap_{P\in \Ass (J)} \hull (I+J^m)_P\cap K[X]$. Thus, for $m \ge \max\{m(P)\mid P\in \Ass (J)\}$, $\hull (I+J^m)$ is an intersection of primary components of a primary decomposition of $I$. 
\end{proof}

	Using variants of  double ideal quotient, we devise a criterion for primary component and generate isolated primary components. We remark that Theorem \ref{primarycri} holds for any Noetherian rings. 
	
	\begin{Theorem}[\cite{Ishi-Yoko}, Theorem 26 (Criterion 1)] \label{primarycri}
		Let $I$ be an ideal and $P$ a prime divisor of $I$. For a $P$-primary ideal $Q$, assume $Q\not \supset (I:P^{\infty})$ and let $J=(I:P^{\infty})\cap Q$. Then, the following conditions are equivalent. 
			
			\begin{itemize}
				\item[$(A)$]  $Q$ is a $P$-primary component for some primary decomposition of $I$.
				\item[$(B)$] $(I:(I:J)^{\infty})=J$. 
			\end{itemize}
	\end{Theorem}
	
	We also generalize Theorem \ref{primarycri} to intersection of primary components as follows. We can check whether $m$ appearing in Lemma \ref{hullpm_gen} is large enough or not by Theorem \ref{primarycri_gen}. 
	
	\begin{Theorem} \label{primarycri_gen}
		Let $I$ be an ideal and $J$ an intersection of prime divisors of $I$. Suppose $J$ is unmixed. For an unmixed ideal $L$ with $\sqrt{L}=J$, assume $\sqrt{(L: (I:J^{\infty}))}=J$ and let $Z=(I:J^{\infty})\cap L$. Then, the following conditions are equivalent. 
			\begin{itemize}
				\item[$(A)$]  $L=\bigcap_{P\in \Ass (J)} Q(P)$ where $Q(P)$ is a $P$-primary components of $I$. 
				\item[$(B)$] $(I:(I:Z)^{\infty})=Z$. 
			\end{itemize}
	\end{Theorem}
	\begin{proof}
		First, we show $(A)$ implies $(B)$. From $(A)$, it is easy to see that $\mathcal{T} =\Ass ((I:J^\infty))\cup \Ass (L)$ is {\em an isolated set} (see Definition 5 in \cite{Ishi-Yoko}). Indeed, for $P^\prime \in \Ass(I)$, if there exists $P\in \mathcal{T}$ s.t. $P^\prime \subset P$, then $P^\prime \in \mathcal{T}$ since $\Ass ((I:J^{\infty}))=\{P^\prime\in \Ass (I)\mid J\not \subset P^\prime\}$ and $\Ass (L)=\Ass (J)$. Thus, for $S=K[X]\setminus ( \bigcup_{P\in \mathcal{T}} P)$, we obtain $Z=IK[X]_S\cap K[X]$ from Lemma 6 in \cite{Ishi-Yoko} and $\mathcal{T}=\Ass (Z)$. By Lemma 25 in \cite{Ishi-Yoko}, we obtain $(I:(I:Z)^{\infty})=Z$. 
		
		Second, we show $(B)$ implies $(A)$. Let $P\in \Ass (J)$. Then, we obtain $P_P=J_P$ and $\sqrt{L_P}=(\sqrt{L})_P=J_P=P_P$. Thus, $L_P$ is a $P_P$-primary ideal and $Z_P=(I:J^{\infty})_P\cap L_P=(I_P:J_P^{\infty})\cap L_P$. Since $\sqrt{(L: (I:J^{\infty}))}=J$, $\sqrt{(L_P: (I_P:J_P^{\infty}))}=P_P$ and thus $L_P\not \supset (I_P:J_P^{\infty})$; otherwise we get $\sqrt{(L_P: (I_P:J_P^{\infty}))}=K[X]_P\neq P_P$. Here, $(I_P:(I_P:Z_P)^{\infty})=Z_P$ for all $P\in \Ass (J)$ since $(I:(I:Z)^{\infty})=Z$ and $(I_P:(I_P:Z_P)^{\infty})=(I:(I:Z)^{\infty})_P$. Thus, by Theorem \ref{primarycri}, $L_P$ is a primary component of $I_P$. Since $L$ is unmixed and $L=\sqrt{J}$, it follows that $L=\bigcap_{P\in \Ass (J)} L_P\cap K[X]$. From Remark \ref{rem:com}, $L_P\cap K[X]$ is a $P$-primary component of $I$ if and only if $L_P$ is a $P_P$-primary component of $I_P$. Finally, we obtain the equivalence. 
	\end{proof}
	
	Also, we can compute the isolated primary component from its associated prime by a variant of double ideal quotient.
	
	\begin{Theorem}[\cite{Ishi-Yoko}, Theorem 36] \label{thm-p}
	Let $I$ be an ideal and $P$ an isolated prime divisor of $I$. Then
	\[
	\hull ((I:(I:P^{\infty})^{\infty}))
	\]
	is the isolated $P$-primary component of $I$. 
	\end{Theorem}
	
	We generalize Theorem \ref{thm-p} as follows. 
	
	\begin{Theorem} \label{thm-p_gen}
	Let $I$ be an ideal and $J$ an intersection of isolated prime divisors of $I$. Suppose $J$ is unmixed. Then
	\[
	\hull ((I:(I:J^{\infty})^{\infty}))=\bigcap_{P\in \Ass (J)} Q(P)
	\]
	where $Q(P)$ is the isolated $P$-primary component of $I$. 
	\end{Theorem}
	
	\begin{proof}
		Let $\mathcal{Q}$ be a primary decomposition of $I$. By Proposition 22 in \cite{Ishi-Yoko}, we obtain
		\[
		(I:(I:J^{\infty})^{\infty})=\bigcap_{Q\in \mathcal{Q}, J\subset \sqrt{IK[X]_{\sqrt{Q}} \cap K[X]}} Q. 
		\]
		Since $J\subset \sqrt{IK[X]_{\sqrt{Q(P)}} \cap K[X]}=\sqrt{Q(P)}=P$ for $P\in \Ass (J)$, it follows that
		\[
		(I:(I:J^{\infty})^{\infty})=\bigcap_{P\in \Ass (J)} Q(P) \cap \bigcap_{Q\in \mathcal{Q}, J\subset \sqrt{IK[X]_{\sqrt{Q}} \cap K[X]}, \sqrt{Q}\not \in \Ass (J)} Q. 
		\]
		As $J$ is unmixed, each $Q(P)$ has the same dimension for $P\in \Ass (J)$. Then, $\dim ( \bigcap_{Q\in \mathcal{Q}, J\subset \sqrt{IK[X]_{\sqrt{Q}} \cap K[X]}, \sqrt{Q}\not \in \Ass (J)} Q)<\dim (J)$ from the fact that for $Q\in \mathcal{Q}$ with $J\subset \sqrt{IK[X]_{\sqrt{Q}} \cap K[X]}$ and $\sqrt{Q}\not \in \Ass (J)$, there exists $P\in \Ass (J)$ s.t. $P\subsetneq \sqrt{Q}$. Since $J$ is an intersection of isolated prime divisors of $I$, we obtain
		\[
		\hull ((I:(I:J^{\infty})^{\infty}))=\bigcap_{P\in \Ass (J)} Q(P). 
		\]
	\end{proof}

\subsection{Modular  techniques for double ideal quotient} \label{sec2-2}

We propose modular techniques for double ideal quotient. For a prime number $p$, let $\mathbb{Z}_{(p)}=\{a/b\in \mathbb{Q}\mid p \nmid b\}$ be the localized ring by $p$ and $\mathbb{F}_p[X]$ the polynomial ring over the finite field. We denote by $\phi_p$  the canonical projection $\mathbb{Z}_{(p)}[X]\to \mathbb{F}_p[X]$. For $F\subset \mathbb{Q}[X]$, we denote by $I(F)$ the ideal generated by $F$. For $F\subset \mathbb{Z}_{(p)}[X]$, we denote $\langle \phi_p(F)\rangle$ by $I_p(F)$ and $\phi_p(I(F)\cap \mathbb{Z}_{(p)}[X])$ by $I_p^0(F)$ respectively. 

We recall the outline of "modular algorithm for ideal operation" (see \cite{Noro-Yoko}) as Algorithm \ref{alg1}. Given ideals $I,J$, ideal operations $AL(*,*)$ over $\mathbb{Q}[X]$ and $AL_p(*,*)$ over $\mathbb{F}_p[X]$ as inputs, we compute $AL(I,J)$ as the output by using modular computations. First, we choose a list of random prime numbers $\mathcal{P}$, which satisfies certain computable condition  \textsc{primeTest}. For example, \textsc{primeTest} is to check whether $p$ is permissible (see Definition \ref{defp}) for Gr\"obner bases of $I$ and $J$ or not. Next, we compute modular operations $H_p=AL_p (I,J)$ for each $p\in \mathcal{P}$. After omitting expected unlucky primes by \textsc{deleteUnluckyPrimes}, we lift $H_p$'s up to $H_{can}$ by CRT and rational reconstruction. Finally, we check $H_{can}$ is really the correct answer by  \textsc{finalTest}. If \textsc{finalTest} says \textsc{False}, then we enlarge $\mathcal{P}$ and continue from the first step. In this paper, we introduce new \textsc{finalTest} for ideal quotient and double ideal quotient. We remark that the termination of this modular algorithm is ensured by the finiteness of unlucky prime numbers. For example, for a given ideals $I$, $J$ and an algorithm for the ideal quotient $(I:J)$ over the rational numbers, there are only finite many steps from the inputs to the outputs and thus the number of coefficients is also finite; hence we can project the computations onto those over finite fields $\mathbb{F}_p$ for all prime numbers $p$ except those appearing in coefficients (see Lemma 6.1 in \cite{Noro-Yoko} for details).  

\begin{algorithm}                      
\caption{Modular Algorithm for Ideal Operation}         
\label{alg1}                          
\begin{algorithmic}                  
\REQUIRE $I,J$: ideals, $AL(*,*)$: an ideal operation over $\mathbb{Q}[X]$, $AL_p(*,*)$: an ideal operation over $\mathbb{F}_p[X]$, 
\ENSURE $AL(I,J)$ over $\mathbb{Q}[X]$
\STATE choose $\mathcal{P}$ as a list of random primes satisfying \textsc{primeTest};
\STATE $\mathcal{HP}=\emptyset $;
\WHILE{}
	\FOR{$p\in \mathcal{P}$}
		\STATE compute $H_p=AL_p(I,J)$;
		\STATE $\mathcal{HP}=\mathcal{HP}\cup \{H_p\}$;
	\ENDFOR
	\STATE $(\mathcal{HP}_{lucky},\mathcal{P}_{lucky})=\textsc{deleteUnluckyPrimes}(\mathcal{HP},\mathcal{P})$;
	\STATE lift $\mathcal{HP}_{lucky}$ to $H_{can}$ by CRT and rational reconstruction;
	\IF{$H_{can}$ passes \textsc{finalTest}}
		\RETURN $H_{can}$
	\ENDIF
	\STATE enlarge $\mathcal{P}$ with prime numbers not used so far;
\ENDWHILE
\end{algorithmic}
\end{algorithm}

First, we introduce some notions of {\em good} primes as follows.  

\begin{Definition}[\cite{Noro-Yoko}, Definition 2.1] \label{defp}

	Let $p$ be a prime number, $F\subset \mathbb{Q}[X]$ and $\prec$ a monomial ordering. Let $G$ be the reduced Gr\"obner basis of $I(F)$ with respect to $\prec$. Here, we denote by $\lc_{\prec}(f)$ the leading coefficient of a polynomial $f$ with respect to $\prec$. 
	\begin{enumerate}
		\item $p$ is said to be {\em weak permissible} for $F$, if $F\subset  \mathbb{Z}_{(p)}[X]$.
		\item $p$ is said to be {\em permissible} for $F$ and $\prec$, if $p$ is weak permissible for $F\subset \mathbb{Q}[X]$ and $\phi_p (\lc_{\prec}(f))\neq 0$ for all $f$ in $F$. 
		\item $p$ is said to be {\em compatible} with $F$ if $p$ is weak permissible for $F$ and $I_p^0(F) = I_p (F)$. 
		\item $p$ is said to be {\em effectively lucky} for $F$ and $\prec$, if $p$ is permissible for $(G,\prec)$ and $\phi_p (G)$ is the reduced Gr\"obner basis of $I_p(G)$.  
	\end{enumerate}
\end{Definition}

\begin{Remark} \label{rem:pcom}
	If $p$ is effectively lucky for $F$ and $\prec$, then $p$ is compatible with $F$ (see Lemma 3.1 (3) in \cite{Noro-Yoko}). 
\end{Remark}

Next, the notion of $p$-compatible Gr\"obner basis candidate is very useful for easily computable tests toward \textsc{finalTest} in modular techniques								. 

\begin{Definition}[\cite{Noro-Yoko}, Definition 4.1] \label{def:pcomm}
	Let $G_{can}$ be a finite subset of\, $\mathbb{Q}[X]$ and $F\subset \mathbb{Q}[X]$. We call $G_{can}$ {\em a $p$-compatible Gr\"obner basis candidate} for $F$ and $\prec$, if $p$ is permissible for $G_{can}$ and $\phi_p(G_{can})$ is a Gr\"obner basis of $I_p^0(F)$ with respect to $\prec$. 
\end{Definition}

The following can be used to \textsc{finalTest} in modular techniques. 

\begin{Lemma}[\cite{Noro-Yoko}, Proposition 4.1] \label{lem:p-com}
	Suppose that $G_{can}$ is a $p$-compatible Gr\"obner basis candidate for $(F,\prec)$, and $G_{can}\subset I(F)$. Then $G_{can}$ is a Gr\"obner basis of $I(F)$ with respect to $\prec$. 
\end{Lemma}

We introduce the following easily computable tests for ideal quotient and saturation in modular techniques, appearing in \cite{Noro-Yoko}. 

\begin{Lemma}[\cite{Noro-Yoko}, Lemma 6.2 and Lemma 6.4] \label{lemma:colon}
	Suppose that a prime number p is compatible with $(F,\prec)$ and permissible for $(f, \prec)$. For a finite subset $H_{can}\subset \mathbb{Q}[X]$, $H_{can}$ is a Gr\"obner basis of $(I(F) : f)$ with respect to $\prec$, if the following conditions hold;
	\begin{enumerate}
		\item $p$ is permissible for $(H_{can},\prec)$,
		\item $\phi_p (H_{can})$ is a Gr\"obner basis of $(I_p(F) : \phi_p(f))$ with respect to $\prec$, 
		\item $H_{can}\subset (I(F) : f)$. 
	\end{enumerate}
	
	For a finite subset $L_{can}\subset \mathbb{Q}[X]$, $L_{can}$ is a Gr\"obner basis of $(I(F) : f^{\infty})$ with respect to $\prec$, if the following conditions hold;
	\begin{enumerate}
		\item $p$ is permissible for $(L_{can},\prec)$,
		\item $\phi_p (L_{can})$ is a Gr\"obner basis of $(I_p(F) : \phi_p(f)^{\infty})$ with respect to $\prec$, 
		\item $L_{can}\subset (I(F) : f^{\infty})$. 
	\end{enumerate}
\end{Lemma}

We generalize Lemma \ref{lemma:colon} by replacing $f$ into an ideal $J$ as follows. We recall that $I_p(G)=\langle \phi_p(G)\rangle_{\mathbb{F}_p[X]}$ where $p$ is weak permissible for $G$. 

\begin{Lemma} \label{lemma:colon_gen}
	Suppose that a prime number p is compatible with $(F,\prec)$ and permissible for $(G, \prec)$. For a finite subset $H_{can}\subset \mathbb{Q}[X]$, $H_{can}$ is a Gr\"obner basis of $(I(F) : I(G))$ with respect to $\prec$, if the following conditions hold;
	\begin{enumerate}
		\item $p$ is permissible for $(H_{can},\prec)$,
		\item $\phi_p (H_{can})$ is a Gr\"obner basis of $(I_p(F) : I_p(G))$ with respect to $\prec$, 
		\item $H_{can}\subset (I(F) : I(G))$. 
	\end{enumerate}
\end{Lemma}

\begin{proof}
	Since $p$ is permissible for $(H_{can},\prec)$, we can consider $I_p(H_{can})=\langle \phi_p(H_{can})\rangle$. It is enough to show $I_p(H_{can})=\phi_p((I(F):I(G)) \cap \mathbb{Z}_{(p)}[X])$ since the equation implies $H_{can}$ is a $p$-compatible Gr\"obner basis candidate for $(I(F):I(G))$ with respect to $\prec$ and a Gr\"obner basis of $(I(F):I(G))$ with respect to $\prec$ from $H_{can}\subset (I(F) : I(G))$ and Lemma \ref{lem:p-com}. 
	
	It is clear that $I_p(H_{can})\subset \phi_p((I(F):I(G)) \cap \mathbb{Z}_{(p)}[X])$ as $H_{can}\subset (I(F) : I(G))$. To show the inverse inclusion, we pick $h\in (I(F):I(G)) \cap \mathbb{Z}_{(p)}[X]$. Then, $hG\subset I(F)\cap \mathbb{Z}_{(p)}[X]$ where $hG=\{hg\mid g\in G\}$ since $p$ is permissible for $h$ and $G$. Thus, 
	\begin{align*}
		  \phi_p(h)I_p(G)&=\phi_p(h)\langle \phi_p(G)\rangle=\langle \phi_p (hG)\rangle\\
		  			& \subset \langle \phi_p(I(F)\cap \mathbb{Z}_{(p)}[X])\rangle=I_p^0(F)=I_p(F)
	\end{align*}
 	by the compatibility of $F$; we obtain $\phi_p(h)\in (I_p(F):I_p(G))=I_p(H_{can})$. Hence $I_p(H_{can})\supset \phi_p((I(F):I(G)) \cap \mathbb{Z}_{(p)}[X])$. 
\end{proof}

\begin{Remark}
	We can check whether $H_{can}\subset (I(F) : I(G))$ or not, by checking whether $I(H_{can})I(G)\subset I(F)$ or not. 
\end{Remark}

We apply this lemma to double ideal quotient as follows. 

\begin{Theorem} \label{theorem:diq}
	Suppose that a prime number p is compatible with $(F,\prec)$ and permissible for $(G, \prec)$. Assume $p$ satisfies $(I_p(F):I_p(G))=\phi_p((I(F):I(G))\cap \mathbb{Z}_{(p)}[X])$. For a finite subset $K_{can}\subset \mathbb{Q}[X]$, $K_{can}$ is a Gr\"obner basis of $(I(F):(I(F):I(G)))$ with respect to $\prec$ if the following conditions hold; 
	
	\begin{enumerate}
		\item $p$ is permissible for $(K_{can},\prec)$,
		\item $\phi_p (K_{can})$ is a Gr\"obner basis of $(I_p(F):(I_p(F):I_p(G)))$ with respect to $\prec$, 
		\item $K_{can}\subset (I(F):(I(F):I(G)))$. 
	\end{enumerate}
\end{Theorem}
\begin{proof}
	Since $p$ is permissible for $(K_{can},\prec)$, we can consider $I_p(K_{can})=\langle \phi_p(K_{can})\rangle$. By Lemma \ref{lem:p-com}, it is enough to show that $K_{can}$ is a $p$-compatible Gr\"obner basis candidate of $(I(F):(I(F):I(G)))$. Since $K_{can}\subset (I(F):(I(F):I(G)))$, $I_p(K_{can})\subset \phi_p( (I(F):(I(F):I(G)))\cap \mathbb{Z}_{(p)}[X])$ holds. Thus, we show the other inclusion. Let $h\in (I(F):(I(F):I(G)))\cap \mathbb{Z}_{(p)}[X]$. Then, 
	\[
	\phi_p (h)\phi_p ((I(F):I(G))\cap \mathbb{Z}_{(p)}[X])\subset \phi_p (I(F)\cap \mathbb{Z}_{(p)}[X])=I_p^0(F)=I_p(F).
	\]
	Since $\phi_p ((I(F):I(G))\cap \mathbb{Z}_{(p)}[X])=(I_p(F):I_p(G))$, we obtain $\phi_p (h)\in (I_p(F):(I_p(F):I_p(G)))=I_p(K_{can})$. Hence, $I_p(K_{can})\supset \phi_p( (I(F):(I(F):I(G)))\cap \mathbb{Z}_{(p)}[X])$. 
\end{proof}

To check the conditions $(I_p(F):I_p(G))=\phi_p((I(F):I(G))\cap Z_{(p)}[X])$ and $K_{can} \subset (I(F):(I(F):I(G)))$, we need a Gr\"obner basis $H$ of $(I(F):I(G))$ in general (the former by $I_p(H)=(I_p(F):I_p(G))$ and the latter by $I(K_{can}) I(H)\subset I(F)$, respectively). However, as to the latter, in a special case that $P$ is an associated prime divisor of $I$, we confirm it more easily. Setting $I(G)=P$ for a prime ideal $P$, we devise the following "Associated Test" using modular techniques. 
 
\begin{Theorem}[Associated Test] \label{assTest}
	Let $I$ be an  ideal and $P$ a prime ideal. Let $F$ and $G$ be Gr\"obner bases of $I$ and $P$ respectively. Suppose $p$ is permissible for $F$, $G$ and satisfies $(I_p(F):I_p(G))=\phi_p((I(F):I(G))\cap \mathbb{Z}_{(p)}[X])$. Let $K_{can}$ be a finite subset of $\mathbb{Q}[X]$. Then, $P$ is a prime divisor of $I$ if the following conditions hold; 
	\begin{enumerate}
		\item $p$ is permissible for $(K_{can},\prec)$,
		\item $\phi_p (K_{can})$ is a Gr\"obner basis of $(I_p(F):(I_p(F):I_p(G)))$ with respect to $\prec$, 
		\item $(I_p(F):(I_p(F):I_p(G)))=I_p(G)$, 
		\item $K_{can}\subset P$. 
	\end{enumerate}
\end{Theorem}

\begin{proof}
	To prove this, we use Theorem \ref{theorem:diq}. If all conditions of Theorem \ref{theorem:diq} hold, then $K_{can}$ is a Gr\"obner basis of $(I:(I:P))$ and thus $(I:(I:P))\subset P$ by the condition $K_{can}\subset P$; hence, $P$ is a prime divisor of $I$ by Proposition \ref{cri1}. Now, we show that all conditions of Theorem \ref{theorem:diq} hold. Since we have directly (1) and (2) in Theorem \ref{theorem:diq}, it is enough to check the condition $K_{can}\subset (I(F):(I(F):I(G)))$. Indeed, we obtain $K_{can}\subset P\subset  (I(F):(I(F):I(G)))$ by Remark \ref{rem:IP} and (4). 
\end{proof}

In above associated test, $K_{can}$ will be $G$ if $P$ is a prime divisor of $I$. Thus, we omit CRT and rational reconstruction as follows. Also, we minimize the number of prime numbers we use since we can check the number is large enough comparing with the following $\|G\|$. For a finite set $G$ of $\mathbb{Q}[X]$, we define 
\[
\|G\|=\max\{a^2+b^2\mid \frac{a}{b} \text{ is a coefficient in a term of an element of } G\}. 
\]

\begin{Corollary}[Associated Test without CRT, Algorithm \ref{alg2}] \label{Cor:Asstest}
		Let $I$ be an  ideal and $P$ a prime ideal. Let $F$ and $G$ be Gr\"obner bases of $I$ and $P$ respectively. Let $\mathcal{P}$ be a finite set of prime numbers. Suppose every $p\in \mathcal{P}$ is permissible for $F$, $G$ and satisfies $(I_p(F):I_p(G))=\phi_p((I(F):I(G))\cap \mathbb{Z}_{(p)}[X])$. Then, $P$ is a prime divisor of $I$ if the following conditions hold; 
	\begin{enumerate}
		\item $(I_p(F):(I_p(F):I_p(G)))=I_p(G)$ for every $p\in \mathcal{P}$, 
		\item $\prod_{p\in \mathcal{P}} p$ is larger than $\|G\|$. 
	\end{enumerate}
\end{Corollary}
\begin{proof}
	Since $\prod_{p\in \mathcal{P}} p$ is larger than coefficients appearing in $G$ for the rational reconstruction (see Lemma 4.2. in \cite{Bohm}), $G$ is a Gr\"obner basis candidate itself and we can set $K_{can}=G$ in Theorem \ref{assTest}. Then, $K_{can}$ satisfies all conditions of the theorem. 
\end{proof}

\begin{algorithm}                      
\caption{Associated Test without CRT}         
\label{alg2}                          
\begin{algorithmic}                  
\REQUIRE $F$: a Gr\"obner basis of an ideal $I$, $G$: a Gr\"obner basis of a prime ideal $P$, $H$: a Gr\"obner basis of $(I(F):I(G))$. 
\ENSURE 
		\textsc{True} if $P$ is a prime divisor of $I$
\STATE choose $\mathcal{P}$ as a list of random primes satisfying \textsc{primeTest} ($p\in \mathcal{P}$ is permissible for $F$, $G$ and $H$) and $\prod_{p\in \mathcal{P}} p> \|G\|$; 
\STATE RESTART;
\WHILE{}
	\FOR{$p\in \mathcal{P}$}
		\IF{$(I_p(F):(I_p(F):I_p(G)))\neq I_p(G)$}
			\STATE delete $p$ from $\mathcal{P}$;
		\ENDIF
	\ENDFOR
	\IF{$\prod_{p\in \mathcal{P}} p\le \|G\|$}
		\STATE enlarge $\mathcal{P}$ with prime numbers not used so far and go back to RESTART; 
	\ENDIF
	\IF{$(I_p(F):I_p(G))=I_p(H)$ for every $p\in \mathcal{P}$}
		\RETURN \textsc{True}
	\ENDIF
	\STATE enlarge $\mathcal{P}$ with prime numbers not used so far and go back to RESTART;
\ENDWHILE
\end{algorithmic}
\end{algorithm}

\begin{algorithm}[H]                      
\caption{Non-Associated Test}         
\label{alg3}                          
\begin{algorithmic}                  
\REQUIRE $F$: a Gr\"obner basis of an ideal $I$, $G$: a Gr\"obner basis of a prime ideal $P$, $H$: a Gr\"obner basis of $(I(F):I(G))$. 
\ENSURE \textsc{False} if $P$ is NOT a prime divisor of $I$
\STATE choose $\mathcal{P}$ as a list of random primes satisfying \textsc{primeTest};
\STATE $\mathcal{KP}=\emptyset $;
\WHILE{}
	\FOR{$p\in \mathcal{P}$}
		\STATE compute $K_p=(I_p(F):(I_p(F):I_p(G)))$;
		\IF{$(I_p(F):(I_p(F):I_p(G)))= I_p(G)$}
			\STATE delete $p$ from $\mathcal{P}$;
		\ELSE
		\STATE $\mathcal{KP}=\mathcal{KP}\cup \{K_p\}$;
		\ENDIF
	\ENDFOR
	\STATE $(\mathcal{KP}_{lucky},\mathcal{P}_{lucky})=\textsc{deleteUnluckyPrimes}(\mathcal{KP},\mathcal{P})$;
	\STATE lift $\mathcal{KP}_{lucky}$ to $K_{can}$ by CRT and rational reconstruction;
	\IF{$I(K_{can})I(H)\subset I$}
		\RETURN \textsc{False}
	\ENDIF
	\STATE enlarge $\mathcal{P}$ with prime numbers not used so far;
\ENDWHILE
\end{algorithmic}
\end{algorithm}

Also, we devise a non-associated test as follows. The test is useful since it does not need a condition $(I_p(F):I_p(G))=\phi_p((I(F):I(G))\cap \mathbb{Z}_{(p)}[X])$. 
\begin{Theorem}[Non-Associated Test, Algorithm \ref{alg3}] \label{nonass}
	Let $I$ be an  ideal and $P$ a prime ideal. Let $F$ and $G$ be Gr\"obner bases of $I$ and $P$ respectively. Suppose $p$ is permissible for $F$ and $G$. Let $K_{can}\subset \mathbb{Q}[X]$ and we assume $p$ is permissible for $K_{can}$. Then, $P$ is not a prime divisor of $I$ if the following conditions hold; 
	\begin{enumerate}
		\item $\phi_p (K_{can})$ is a Gr\"obner basis of $(I_p(F):(I_p(F):I_p(G)))$ with respect to $\prec$, 
		\item $K_{can}\subset (I:(I:P))$,
		\item $(I_p(F):(I_p(F):I_p(G)))\neq I_p(G)$. 
	\end{enumerate}
\end{Theorem}

\begin{proof}
	Suppose $P$ is a prime divisor of $I$. Then, $(I:(I:P))=P$ from Remark \ref{rem:IP} and  
	\begin{align*}
		\phi_p(K_{can})&\subset \phi_p((I:(I:P)) \cap \mathbb{Z}_{(p)}[X])\\ 
		&\subset \phi_p(P\cap \mathbb{Z}_{(p)}[X])=I^0_p(G)=I_p(G). 		
	\end{align*}
	
	Since $\langle \phi_p(K_{can}) \rangle = (I_p(F):(I_p(F):I_p(G)))\supset I_p(G)$, we obtain $(I_p(F):(I_p(F):I_p(G))= I_p(G)$. This contradicts $(I_p(F):(I_p(F):I_p(G)))\neq I_p(G)$. 
\end{proof} 

Next, we consider modular saturation. Since $(I:J^m)=(I:J^{\infty})$ for a sufficiently large $m$, the following holds from Lemma \ref{lemma:colon_gen}. 

\begin{Lemma} \label{lemma:sat_gen}
	Suppose that a prime number p is compatible with $(F,\prec)$ and permissible for $(G, \prec)$. For a finite subset $H_{can}\subset \mathbb{Q}[X]$, $H_{can}$ is a Gr\"obner basis of $(I(F) : I(G)^{\infty})$ with respect to $\prec$, if the following conditions hold;
	\begin{enumerate}
		\item $p$ is permissible for $(H_{can},\prec)$,
		\item $\phi_p (H_{can})$ is a Gr\"obner basis of $(I_p(F) : I_p(G)^{\infty})$ with respect to $\prec$, 
		\item $H_{can}\subset (I(F) : I(G)^{\infty})$. 
	\end{enumerate}
\end{Lemma}

To check $H_{can}\subset (I(F) : I(G)^{\infty})$, we can use the following. 

\begin{Lemma} \label{lem:2213}
	Let $H_{can},F$ and $G$ be finite sets of $K[X]$. For $G=\{f_1,\ldots,f_k\}$ and a positive integer $m$, we denote $\{f_1^m,\ldots,f_k^m\}$ by $G^{[m]}$. Then, the following conditions are equivalent. 
	
			$(A)$ $H_{can}\subset (I(F) : I(G)^{\infty})$,
			
			$(B)$ $I(H_{can})I(G)^m\subset I(F)$ for some $m$,
			
			$(C)$ $I(H_{can})I(G^{[m]})\subset I(F)$ for some $m$.
\end{Lemma}

\begin{proof}
	$[(A)\Rightarrow(B)]$ This is obvious from the definition of $(I(F) : I(G)^{\infty})$. $[(B)\Rightarrow(C)]$ Since $I(G^{[m]})\subset I(G)^m$, $I(H_{can})I(G^{[m]})\subset I(H_{can})I(G)^m\subset I(F)$.  $[(C)\Rightarrow(A)]$ As $I(G)^{km}\subset I(G^{[m]})$, we obtain $I(H_{can})I(G)^{km}\subset I(H_{can})I(G^{[m]})\subset I(F)$ and $H_{can}\subset (I(F) : I(G)^{\infty})$. 
\end{proof}

Since  the number of generators of $I(G^{[m]})$ is less than that of $I(G)^m$, it is better to check whether $I(H_{can})I(G^{[m]})\subset I(F)$ or not. 

Finally, we introduce modular techniques for double saturation. 

\begin{Theorem} \label{theorem:dbsat}
	Suppose that a prime number p is compatible with $(F,\prec)$ and permissible for $(G, \prec)$. Assume $p$ satisfies $(I_p(F):I_p(G)^{\infty})=\phi_p((I(F):I(G)^{\infty})\cap \mathbb{Z}_{(p)}[X])$. For a finite subset $K_{can}\subset \mathbb{Q}[X]$, $K_{can}$ is a Gr\"obner basis of $(I(F):(I(F):I(G)^{\infty})^{\infty})$ with respect to $\prec$ if the following conditions hold; 
	
	\begin{enumerate}
		\item $p$ is permissible for $(K_{can},\prec)$,
		\item $\phi_p (K_{can})$ is a Gr\"obner basis of $(I_p(F):(I_p(F):I_p(G)^{\infty})^{\infty})$ with respect to $\prec$, 
		\item $K_{can}\subset (I(F):(I(F):I(G)^{\infty})^{\infty})$. 
	\end{enumerate}
\end{Theorem}

\begin{proof}
		For a sufficiently large integer $m$, $(I(F):I(G)^\infty)=(I(F):I(G)^m)$ and $(I_p(F):I_p(G)^\infty)=(I_p(F):I_p(G)^m)$. Thus, we can prove this by the similar way of Theorem \ref{theorem:diq}. 
\end{proof}

\subsection{Intermediate primary decomposition} \label{sec:ipd}

In this section, we introduce intermediate primary decomposition as a bi-product of modular localizations devised in Section \ref{sec2-2}. We give a rough outline of possible "intermediate primary decomposition via MIS". In general, modular primary decomposition is very difficult to compute since primary component may be different over infinite many finite fields. For example, $I=(x^2+1)\cap (x+1)$ is a primary decomposition in $\mathbb{Q}[X]$, however, it is not one in $\mathbb{F}_p[X]$ for every prime number $p$ of type $p=4n+1$. Thus, we propose {\em intermediate primary decomposition via MIS} instead of full primary decomposition. For a subset of variables $X$ and an ideal $I$, we call $U$ a maximal independent set (MIS) of $I$ if $K[U]\cap I=\{0\}$ (see Definition 3.5.3 in \cite{greuel2002singular}). Then, for a subset $U\subset X$, we define 
\[
\Ass_U(I_p(F))=\{\bar{P}_p\in \Ass (I_p(F))\mid U\textit{ is a MIS of $\bar{P}_p$}\}.
\] 
where $p$ is permissible for $F$. Also, we denote the set of prime divisors of $I$ which have the same MIS $U$ by 
\[
\Ass_U(I)=\{P\in \Ass (I)\mid U\textit{ is a MIS of $P$}\}.
\] 

We note that $U$ is a MIS of $I(F)$ if $U$ is one of the initial ideal $in_{\prec} (I(F))$ (see Exercise 3.5.1 in \cite{greuel2002singular}). Thus, if $p$ is effective lucky for ($F$,$\prec$) and $U$ is a MIS of $in_\prec (I(F))$ then $U$ is also a MIS of $I(F)$ and $I_p(F)$. Here, we define intermediate primary decomposition in general setting as follows (a certain generalization of one in \cite{SHIMOYAMA1996247}). 

\begin{Definition} \label{def:ipd}
	Let $I$ be an ideal. Then, a set of ideals $\mathcal{Q}$ is called an {\em intermediate primary decomposition (IPD)} of $I$ if 
	\begin{enumerate}
		\item[(a)] for all $Q\in \mathcal{Q}$, $\Ass (Q)\subset \Ass (I)$,  
		\item[(b)] $\bigcap_{Q\in \mathcal{Q}} Q=I$. 
	\end{enumerate}
	We call $Q\in \mathcal{Q}$ an intermediate primary component of $I$. In particular, when  there is a subset $U$ of $X$ s.t. $\Ass(Q)=\Ass_U(I)$, we call $Q$ an intermediate component of $I$ via $U$. 
\end{Definition}

	We remark that $\bigcup_{Q\in \mathcal{Q}} \Ass (Q)=\Ass (I)$. For computing intermediate primary decomposition, the following Corollary is very useful to generate prime divisors. 

\begin{Corollary} \label{interpri}
	Let $F$ be a Gr\"obner basis of $I$ and $p$ a  permissible prime number for $F$. Let $U$ be a subset of $X$ such that $\Ass_U(I_p(F))$ is not empty, and $\bar{H}$ a Gr\"obner basis of $\bar{J}=\bigcap_{P_p \in \Ass_U(I_p(F))} P_p$. Let  $H_{can}$ be a Gr\"obner basis candidate constructed from $\bar{H}$ and $J=I(H_{can})$. Assume $p$ is permissible for $H_{can}$. Suppose $H_{can}$ is a Gr\"obner basis of $J$ and $p$ is effectively lucky for the reduced Gr\"obner basis $L$ of $(I:J)$ with $ I_p(L)=(I_p(F):I_p(H_{can}))$. If $J$ is a prime ideal then $J$ is a prime divisor of $I$. 
\end{Corollary}

\begin{proof}
	To apply Theorem \ref{assTest} for $I$ and $J$, we check the conditions. First, since $p$ is effectively lucky for $L$, $p$ is compatible with $L$ by Remark \ref{rem:pcom}. Thus, $\phi_p((I(F):I(H_{can}))\cap \mathbb{Z}_{(p)}[X])=I_p^0(L)=I_p(L)=(I_p(F):I_p(H_{can}))$. From the assumption, $p$ is permissible for $H_{can}$. As $I_p(H_{can})=\bar{J}$ is an intersection of prime divisors of $I_p(F)$, it follows that $(I_p(F):(I_p(F):I_p(H_{can}))=I_p(H_{can})$ by Theorem \ref{cri:rad}. Thus,  $\phi_p(H_{can})=\bar{H}$ is a Gr\"obner basis of $(I_p(F):(I_p(F):I_p(H_{can}))$. It is obvious that $H_{can}\subset J$. Hence, all conditions in Theorem \ref{assTest} hold and thus $J$ is a prime divisor of $I$. 
\end{proof}

When $J$ is not prime, we can check the radicality of $J$ by the following lemma. For any effectively lucky $p$ for $H_{can}$, if $\langle \bar{H} \rangle$ is radical then $\langle H_{can} \rangle$ is also radical.  

\begin{Lemma}[\cite{Noro-Yoko}, Lemma 6.7] \label{lemrad}
	Suppose that $H_{can}$ is the output of our CRT modular computation, that is, it satisfies the following: 
	\begin{enumerate}
		\item $p$ is permissible for $(H_{can}, \prec)$,
		\item  $\phi_p (H_{can})$ coincides with the reduced Gr\"obner basis of $\sqrt{I_p(F)}$
		\item $H_{can} \subset \sqrt{I(F)}$
	\end{enumerate}
	Then $H_{can}$ is the reduced Gr\"obner basis of $\sqrt{I(F)}$ with respect to $\prec$. 
\end{Lemma}

We can extend Corollary \ref{interpri} to intersection of prime divisors by using Theorem \ref{cri:rad} as Proposition \ref{interrad}. We can ensure that the lifted ideal $I(H_{can})$ is radical from Lemma \ref{lemrad} and an intersection of prime divisors $I$ from Theorem \ref{cri:rad} and Theorem \ref{theorem:diq}.  

\begin{Proposition} \label{interrad}
Under the conditions of Corollary \ref{interpri} (except the primality of $J$), if $J$ is a radical ideal then $J$ is some intersections of prime divisors of $I$. 
\end{Proposition}

We note that, if $\Ass_U(I_p(F))$ consist of one prime, that is, $\bar J$ is prime, then we check if $J$ is prime or not more easily. Moreover, if $\Ass_U(I_p(F))$ consists of two prime ideals $\bar{P}_1$ and $\bar{P}_2$ and  then we combine those prime divisors and apply the criterion for radical to the lifting of $\bar{P}_1\cap \bar{P}_2$. We also make the same argument for $\bar{P}_1\cap \bar{P}_2\cap \bar{P}_3$, $\bar{P}_1\cap \bar{P}_2\cap \bar{P}_3\cap \bar{P}_4$ and so on. 

\begin{Example}
	Let $I=(x)\cap (x^3,y)\cap (x^2+1)$ Let $F=\{x^3y+xy,x^5+x^3\}$ be the reduced Gr\"obner basis of $I$. We  consider two prime numbers $p=3, 5$. Then, $\Ass (I_3(F))=\{(x), (x,y), (x^2+1)\}$ and $\Ass (I_5(F))=\{(x), (x,y), (x+2),(x+3)\}$. For $U_1=\{y\}$ and $U_2=\emptyset$, $\Ass_{U_1} (I_3(F))=\{(x),(x^2+1)\}$ and $\Ass_{U_2} (I_3(F))=\{(x,y)\}$. Similarly, $\Ass_{U_1} (I_5(F))=\{(x),(x+2),(x+3)\}$ and $\Ass_{U_2} (I_5(F))=\{(x,y)\}$. For $J_p(U)=\bigcap_{P_p\in \Ass_U(I_p(F)) }P_p$, it follows that $J_3(U_1)=(x^3+x)$, $J_5(U_1)=(x^3+x)$, $J_3 (U_2)=(x,y)$ and $J_5(U_2)=(x,y)$. By using CRT, we may compute radicals of intermediate primary components $J_{can}(U_1)=(x^3+x)$ and $J_{can}(U_2)=(x,y)$. Finally, we obtain an intermediate primary decomposition $\{(x^3+x), (x^3,y)\}$ of $I$ from Lemma \ref{hullpm_gen} and Theorem \ref{thm-p_gen}. 
\end{Example}

Finally, we sketch an outline of intermediate primary decomposition via MIS as follows. Its termination comes from the finiteness of unlucky primes for computation of associated prime divisors and primary components.
	
\medskip
\noindent
{\bf Intermediate Primary Decomposition via MIS}
	
	\begin{itemize}
		\item[Input:] $F$: a Gr\"obner basis of an ideal $I$.
		\item[Output:] $\{Q(U)\}$: an IPD via MIS of $I$. 
	\end{itemize}
	
		\begin{itemize}
			\item[(Step 1)] choose $\mathcal{P}$ as a list of random primes satisfying \textsc{primeTest}
			\item[(Step 2)] compute $\Ass(I_p(F))$ for $p\in \mathcal{P}$ and choose a set of MISs $\mathcal{U}$ from $\Ass(I_p(F))$
			\item[(Step 3)] compute $J_p(U)=\bigcap_{P_p\in \Ass_U(I_p(F)) }P_p$ for each $U\in \mathcal{U}$ and let $\mathcal{JP}(U)=\mathcal{JP}(U)\cup \{J_p(U)\}$
			\item[(Step 4)] delete unlucky $p$ for $\mathcal{JP}(U)$ and obtain $\mathcal{JP}_{lucky}(U)$
			\item[(Step 5)] lift $\mathcal{JP}_{lucky}(U)$ to $J_{can}(U)$ by CRT and rational reconstruction. If $J_{can}(U)$ is unmixed then go to Step 6; otherwise RESTART
			\item[(Step 6)] if $J_{can} (U)$ passes \textsc{finalTest}  (Proposition \ref{interrad}) then go to Step 7: otherwise RESTART
			\item[(Step 7)]  compute an intersection of primary components $Q(U)$ by $\hull (I+J_{can}(U)^m)$ (Lemma \ref{hullpm_gen} and \ref{primarycri_gen}) or $\hull ((I:(I:J_{can}(U)^{\infty})^{\infty}))$ (Theorem \ref{thm-p_gen}) for isolated cases
			\item[(Step 8)] if $\bigcap_{U\in \mathcal{U}} Q(U)=I$ then return $\{Q(U)\}$; otherwise RESTART
		\end{itemize}
		RESTART: enlarge $\mathcal{P}$ with prime numbers not used so far and go back to Step 2

\section{Experiments}
In this section, we see some naive experiments on \textsc{Singular} \cite{DGPS}. Timings (in seconds) are measured in real time and on a PC with Intel Core i7-8700B CPU with 32GB memory. We see several examples with intermediate coefficient growth. The source code for several algorithms ( {\tt modQuotient}, {\tt modSat} and {\tt modDiq}) is open in https://github.com/IshiharaYuki/moddiq. 

To implement modular algorithms for (double) ideal quotient and saturation, we use the library {\tt modular.lib}. A function {\tt modular} returns a candidate from modular computations by CRT and rational reconstruction. As the optional arguments, the function has {\tt primeTest},  {\tt deleteUnluckyPrimes},  {\tt pTest}  and  {\tt finalTest}. In this paper, we implemented  {\tt primeTest}, {\tt pTest}  and  {\tt finalTest} for ideal quotient and saturation. Also, we use Singular implemented functions {\tt quotient} and  {\tt sat} to compute $(I:J)$ and $(I:J^{\infty})$ respectively (about computations of ideal quotient and saturation, see \cite{greuel2002singular}). We explain some details of our implementations. First, {\tt modQuotient} computes ideal quotient by modular techniques based on Lemma \ref{lemma:colon_gen}. Second, {\tt modSat} computes saturation by modular techniques based on Lemma \ref{lemma:sat_gen} and Lemma \ref{lem:2213}. Third, {\tt diq} computes double ideal quotient by using {\tt quotient} twice and {\tt modDiq} computes double ideal quotient based on Theorem \ref{theorem:diq}. The function {\tt modDiq} uses {\tt modQuotient} to check the condition that $(I_p(F):I_p(G))=\phi_p((I(F):I(G))\cap Z_{(p)}[X])$ and $K_{can} \subset (I(F):(I(F):I(G)))$ in Theorem \ref{theorem:diq}.  Of course, we can compute double ideal quotient by using {\tt modQuotient} twice. 

 Here, we use the degree reverse lexicographical ordering ({\tt dp} on \textsc{Singular}) .We tested our implementation by "cyclic ideal", where $cyclic(n)$ is defined in $\mathbb{Q}[x_1,\ldots,x_n]$ (see the definition in \cite{Backelin}). We let $P_1=(-15x_{5}+16x_{6}^3-60x_{6}^2+225x_{6}-4,2x_{5}^2-7x_{5}+2x_{6}^2-7x_{6}+28,(4x_{6}-1)x_{5}-x_{6}+4,4x_{1}+x_{5}+x_{6},4x_{2}+x_{5}+x_{6},4x_{3}+x_{5}+x_{6},4x_{4}+x_{5}+x_{6})$ and $P_2=(x_2^2+4x_2+1,x_1+x_2+4,x_3-1,x_4-1,x_5-1,x_6-1)$ be prime divisors of $cyclic(6)$. Let $Q_1=((-15x_{5}+16x_{6}^3-60x_{6}^2+225x_{6}-4)^2,(2x_{5}^2-7x_{5}+2x_{6}^2-7x_{6}+28)^2,(4x_{6}-1)x_{5}-x_{6}+4,4x_{1}+x_{5}+x_{6},4x_{2}+x_{5}+x_{6},4x_{3}+x_{5}+x_{6},4x_{4}+x_{5}+x_{6})$ be a $P_1$-primary ideal. Also, we let $I_1=(8x^2y^2 + 5xy^3 + 3x^3z + x^2yz,x^5 + 2y^3z^2 + 13y^2z^3 + 5yz^4 ,8x^3 + 12y^3 + xz^2, 7x^2y^4 + 18xy^3z^2 + y^3z^3)$ be a modification of an ideal appeared in \cite{Arnold} and $I_2=(xw_{11}-yw_{10},yw_{12}-zw_{11},-w_{11}w_{20}+w_{21}w_{10},-w_{21}w_{12}+w_{22}w_{11})$ be $A_{2,3,3}$ (see \cite{strumfels}). As inputs, we used their Gr\"obner bases. 

In Table \ref{table1}, we can see that {\tt modQuotient} is very effective for computation of such ideals. In table \ref{table2}, we compare timings of computations of saturation in each method. To consider ideals with non-prime components, we take {an intersection} or products of ideals. We can see that modSat is very effective even when multiplicities of target primary components are large.  In table \ref{table3}, we see results of prime divisors checks by double ideal quotient in each method. We can see that modular methods "double {\tt modQuotient}" and {\tt modDiq} are very efficient, comparing with the rational {\tt diq}. In almost cases in the table, {\tt modDiq} is faster than {\tt modQuotient} since the final test (Theorem \ref{theorem:diq}) may have some effectiveness for efficient computations. 

As a whole, we examined the efficiency of modular techniques for ideal quotients by computational experiments. 

\begin{scriptsize}

\begin{table}[H]
		\begin{center}
\begin{tabular}{c|c||c}\hline
ideal quotient & {\tt quotient} & {\tt modQuotient} \\\hline
$(cyclic(6):P_1)$ & 35.0 &11.2 \\ 
$(cyclic(6):P_2)$ &  15.1&7.65 \\ 
$(I_1^2:I_1)$ & 7.80 & 0.32 \\ 
$(I_1^3:I_1)$ &255  &  7.67 \\ 
$(I_1^4:I_1)$ &2137&  68.8 \\ 
$(I_1I_2:I_2)$ & 0.88& 0.72 \\\hline 
\end{tabular}
			\caption{Ideal quotient} \label{table1}
		\end{center} 
	\end{table}
		
	\vspace{-1cm}
	
	\begin{table}[H]
		\begin{center}
\begin{tabular}{c|c||c}\hline
saturation & {\tt sat} & {\tt modSat} \\\hline
$((cyclic(6)\cap Q_1):P_1^{\infty})$ & 86.9&16.4  \\  
$(I_1I_2^2:I_2^{\infty})$ &  1264&21.9  \\ 
$((I_1\cdot (x^{100},xy)):(x,y)^{\infty})$ &0.33&0.13  \\ 
$((I_1\cdot (x^{500},xy)):(x,y)^{\infty})$ &27.3&1.18 \\ 
$((I_1\cdot (x^{1000},xy)):(x,y)^{\infty})$ &201&4.25 \\ \hline 
\end{tabular}
			\caption{Saturation} \label{table2}
		\end{center} 
	\end{table}
	
	\vspace{-1cm}
	
	\begin{table}[H]
		\begin{center}
\begin{tabular}{c|c||c|c}\hline
[ideal, prime divisor] & {\tt diq} & double {\tt modQuotient} & {\tt modDiq} \\\hline
$[cyclic(6), P_1]$ &37.0  &28.9&  17.8\\ 
$[cyclic(6), P_2]$ & 15.3&9.36&11.3 \\ 
$[I_1^3, (x,y)]$ &13.1  &8.96 &5.32 \\ 
$[I_1^4, (x,y)]$ &254  & 81.7&41.4 \\ 
$[I_1^2I_2, (x,y,z)]$ &143 &80.7 &29.1 \\ \hline
\end{tabular}
			\caption{Double ideal quotient} \label{table3}
		\end{center} 
	\end{table}	
	
	\vspace{-1cm}	
	\end{scriptsize}

\section{Conclusion and Remarks}
In this paper, we apply modular techniques to effective localization and double ideal quotient. Double ideal quotient and its variants are used to prime divisor check and generate primary component. Modular techniques can avoid intermediate coefficient growth and thus we can compute double ideal quotient and its variants efficiently. We also devise new algorithms for modular prime divisor check and intermediate primary decomposition. We have already implemented {\tt modQuotient}, {\tt modSat} and {\tt modDiq} on Singular, and we can see that modular techniques are very effective for several examples in experiments. 

We are on the way to implement Associated Check (Algorithm \ref{alg2}, \ref{alg3}) and complete an efficient algorithm of Intermediate Primary Decomposition via MIS. In particular, we can expect that Algorithm \ref{alg2} will also be efficient for examples we see in the experiments of {\tt modDiq}. Combining Algorithm \ref{alg2} and Algorithm \ref{alg3}, we may have a new test for prime divisors as follows. First, we choose a list of random primes $\mathcal{P}$ and check $(I_p(F):(I_p(F):I_p(G)))=I_p(G)$ for each $p\in \mathcal{P}$, where $I(F)$ is an ideal and $I(G)$ is a prime ideal as inputs. Second, if prime numbers s.t. $(I_p(F):(I_p(F):I_p(G)))=I_p(G)$ are majority then we go to  Algorithm \ref{alg2}; otherwise, go to  Algorithm  \ref{alg3}. Finally, we continue to enlarge $\mathcal{P}$ until we pass the associated test (Corollary \ref{Cor:Asstest}) or the non-associated test (Theorem \ref{nonass}). Also, we can compute a Gr\"obner basis of $(I(F):I(G))$ at the same time during the algorithms. 

As our future work, we continue to improve the implementations and extend experiments to other examples. Also, we are thinking about intermediate primary decomposition in another way e.g. by double saturation.

\section*{Acknowledgements}
This work has been advanced during the author's research stay at Technische Universit\"at Kaiserslautern, supported by Overseas Challenge Program for Young Researchers of Japan Society for the Promotion of Science. The author is very grateful to the \textsc{Singular} team for fruitful discussions and kind hospitality there. In particular, he is very thankful to Wolfram Decker and Hans Sch\"onemann for helpful advice of modular techniques and programming on \textsc{Singular} at Kaiserslautern. He appreciates the kind support of the computational facility by Masayuki Noro. He would like to thank his supervisor, Kazuhiro Yokoyama, for constructive comments and suggestions for the paper.

\bibliographystyle{ACM-Reference-Format}
\bibliography{sample-base}

\begin{thebibliography}{99}
	\bibitem{atiyah1994introduction} Atiyah, M.F., MacDonald, I.G.: Introduction to Commutative Algebra. Addison-Wesley Series in Mathematics. Avalon Publishing, New York (1994) 
	\bibitem{Afzal}Afzal, D., Kanwal, F., Pfister, G., Steidel, S.: Solving via Modular Methods. In: Bridging Algebra, Geometry, and Topology, Springer Proceedings in Mathematics \& Statistics, vol. 96, 1-9 (2014)
	\bibitem{Arnold}Arnold, E.: Modular algorithms for computing Gr\"obner bases. J. Symb. Comput. 35, 403-419 (2003)
	\bibitem{Backelin} Backelin, J., Fr\"oberg, R. How we prove that there are exactly 924 cyclic 7-roots. In: Proceedings of ISSAC 91, ACM Press, 103-111 (1991)
	\bibitem{Bohm}B\"ohm, J., Decker, W., Fieker, C., Pfister, G.: The use of bad primes in rational reconstruction. Math. Comput. 84, 3013-3027 (2015)
	\bibitem{DGPS}
Decker, W.; Greuel, G.-M.; Pfister, G.; Sch{\"o}nemann, H.: 
\newblock {\sc Singular} {4-1-2} --- {A} computer algebra system for polynomial computations.
\newblock {http://www.singular.uni-kl.de} (2019).
	\bibitem{Eisenbud1992} Eisenbud, D., Huneke, C., Vasconcelos, W.: Direct methods for primary decomposition. Inventi. Math.110 (1), 207-235 (1992)
	\bibitem{GIANNI1988149}Gianni, P., Trager, B., Zacharias, G.: Gr\"obner bases and primary decomposition of polynomial ideals. J. Symb. Comput. 6(2), 149-167 (1988)
	\bibitem{greuel2002singular} Greuel,  G.-M.,  Pfister,  G.:  A  Singular  Introduction  to  Commutative  Algebra. Springer, Heidelberg (2002). https://doi.org/10.1007/978-3-662-04963-1
	\bibitem{Idrees} Idrees, N., Pfister, G., Steidel, S.: Parallelization of modular algorithms. J. Symb. Comput. 46, 672-684 (2011)
	\bibitem{Ishi-Yoko} Ishihara Y., Yokoyama K.:  Effective Localization Using Double Ideal Quotient and Its Implementation. In: Computer Algebra in Scientific Computing CASC 2018, LNCS, vol. 11077, Springer, pp.272-287 (2018)
	\bibitem{KAWAZOE20111158} Kawazoe, T., Noro, M.: Algorithms for computing a primary ideal decomposition without producing intermediate redundant components. J. Symb. Comput. 46(10), 1158-1172 (2011)
	\bibitem{matzat2012algorithmic} Matzat, B.H., Greuel, G.-M., Hiss, G.: Primary decomposition: algorithms and comparisons. In: Matzat, B.H., Greuel, G.M., Hiss, G. (eds.) Algorithmic Algebra and Number Theory, pp. 187-220. Springer, Heidelberg (1999). https://doi.org/ 10.1007/978-3-642-59932-3
	\bibitem{Noro2009} Noro, M.: Modular algorithms for computing a generating set of the syzygy module. In: Computer Algebra in Scientific Computing CASC 2009, LNCS, vol. 5743, pp. 259-268. Springer (2009)
	\bibitem{Noro-Yoko} Noro, M., Yokoyama, K. Usage of Modular Techniques for Efficient Computation of Ideal Operations. Math.Comput.Sci. 12(1): 1-32, (2018)
	\bibitem{SHIMOYAMA1996247} Shimoyama, T., Yokoyama, K.: Localization and primary decomposition of polynomial ideals. J. Symb. Comput. 22(3), 247-277 (1996) 
	\bibitem{strumfels} Sturmfels, B.: Solving systems of polynomial equations. In: CBMS Regional Conference Series. American Mathematical Society, no. 97 (2002) 	
	\bibitem{vasconcelos2004computational} Vasconcelos, W.: Computational Methods in Commutative Algebra and Algebraic Geometry. Algorithms and Computation in Mathematics. Springer, Heidelberg (2004)
\end{thebibliography}

\end{document}